\documentclass[11pt,a4paper]{amsart}

\usepackage{amsthm} 	%kind of obvious
\usepackage{amssymb} 	%mathcal, mathbb
\usepackage{amsmath} 	%A lot of important stuff
\usepackage{paralist} 	%compactenum
\usepackage[utf8]{inputenc} %for special characters (Umlaute etc.)
\usepackage{hyperref}   %links
\usepackage{cleveref}
\usepackage[backend=biber,style=numeric,citestyle=numeric]{biblatex}
\usepackage{xcolor}
\theoremstyle{definition}
\newtheorem{defn}{Definition}[section]
\crefname{defn}{Definition}{Definitions}

\theoremstyle{remark}
\newtheorem{rmk}[defn]{Remark}
\crefname{rmk}{Remark}{Remarks}

\theoremstyle{plain}
\newtheorem{thm}[defn]{Theorem}
\crefname{thm}{Theorem}{Theorems}

\newtheorem{prop}[defn]{Proposition}
\crefname{prop}{Proposition}{Propositions}
\newtheorem{lem}[defn]{Lemma}
\crefname{lem}{Lemma}{Lemmas}
\newtheorem{cor}[defn]{Corollary}
\crefname{cor}{Corollary}{Corollaries}
\newtheorem{thmx}{Theorem}
\crefname{thmx}{Theorem}{Theorems}

\DeclareMathOperator{\Gal}{Gal}

\DeclareMathOperator{\Qp}{\mathbb{Q}_p}
\DeclareMathOperator{\FF}{\mathbb{F}}
\DeclareMathOperator{\ZZ}{\mathbb{Z}}
\DeclareMathOperator{\CC}{\mathbb{C}}
\DeclareMathOperator{\NN}{\mathbb{N}}
\DeclareMathOperator{\QQ}{\mathbb{Q}}

\DeclareMathOperator{\OO}{\mathcal{O}}

\DeclareMathOperator{\ord}{ord}

\begin{document}

\title[Roots of unity on the graphs of non-Archimedean Laurent series]{Counting roots of unity on the graphs of Laurent series over non-Archimedean local fields}
\author{Christoph P\"utz}
%\author{Harry Schmidt}
\begin{abstract}
We completely classify Laurent series converging on the unit circle over a non-Archimedean local field (of any characteristic) that map infinitely many roots of unity to roots of unity.
For a given Laurent series $f$ over a field of positive characteristic with residue field $\FF_q$, we prove effective bounds for the number of possible roots of unity in terms of the number of zeroes of the auxilliary function $f(x^q) -f(x)^q$ on the unit circle. In characteristic 0 our bound is still effective but also depends on the ramification degree of the base field over $\QQ_p$ as well as the size of the coefficients of $f$. This has applications to the Manin-Mumford conjecture in $\mathbb{G}_m^2$. In characteristic $0$, this work builds upon a pigeon-hole based method by Schmidt.

\end{abstract}

\maketitle
\section{Introduction}

\cref{thma} proves a theorem of Manin-Mumford type, asserting that the number of torsion points on the graph of a one-variable Laurent series defined over a non-Archimedean local field is finite unless the series is special. 
The same strategy yields \cref{thm2,thm3}, providing effective bounds for the cardinality of the set of such torsion points.
The proof relies only on tools from the ramification theory of non-Archimedean local fields and elementary non-Archimedean analysis.

Manin and Mumford independently conjectured the following \cite{lang_65}:
\emph{A curve of genus $g\geq 2$ defined over $\CC$ embedded into its Jacobian has only finitely many torsion points (with respect to the group structure on the Jacobian)?}

The first complete proof of a generalization of this conjecture over fields of characteristic $0$ was given by Raynaud \cite{raynaud}.
The Mordell conjecture claims that \emph{a curve of genus $g\geq 2$ defined over a number field $K$ has only finitely many $K$-rational points.}
Lang unified both of these conjectures in the \emph{Mordell-Lang conjecture} \cite{lang_65}.
This conjecture was settled by Faltings \cite{faltings} for abelian varieties and extended to semi-abelian varieties by McQuillan \cite{mcquillan}, following work of Vojta.

However, there were many earlier proofs for special cases. 
The first were independently given by Ihara, Serre and Tate for a curve in $\mathbb{G}_m^2$.
While this is not a case of the original formulation of the Manin-Mumford conjecture, it deals with the 'torsion part' of the Mordell-Lang conjecture and not with the 'Mordell part' and is thus usually associated with the Manin-Mumford conjecture.
Ihara, Serre and Tate independently proved the following Theorem.

\begin{thm}[Ihara, Serre and Tate, \cite{lang_65}]
		  \label[thm]{thm:ist}
		  Let $K$ be a number field and $P\in K[X,Y]$ be a polynomial such that 
		  $$\{(\zeta,\xi)\in\mu_\infty^2:P(\zeta,\xi)=0\}$$
		  is infinite.
		  Then there are coprime non-negative integers $r,s$ such that $X^rY^s-1$ or $X^r-Y^s$ are not coprime to $P(X,Y)$.
\end{thm}

Setting $P(X,Y):=F(X)-Y$, the following is an immediate consequence.

\begin{cor}
		  \label[cor]{cor:complex_case}
		  Let $K$ be a number field and $F\in K[X]$ be a polynomial such that
		  $$\{(\zeta,\xi)\in\mu_\infty^2:F(\zeta)=\xi\}$$
		  is infinite.
		  Then there is a root of unity $\eta\in\mu_\infty$ and a non-negative integer $m$ such that $F(X)=\eta X^m$.
\end{cor}

Formulating the Manin-Mumford or Mordell-Lang conjectures in positive characteristic is more difficult and needs to account for a bigger class of exceptions.
		  Indeed, let $K$ be a field of characteristic $p$.
		  An algebraic closure $\overline K$ then contains $\overline\FF_p$.
		  Assume that $q$ is a power of $p$ and $f\in\FF_q[X,X^{-1}]$. 
Then $f(x)\in\overline\FF_q$ for all $x\in\overline\FF_q$ and for infinitely many $x$ the value $f(x)$ is non-zero.
But every $x\in\overline\FF_q^{\times}$ is a root of unity.
Versions of the Mordell-Lang or Manin-Mumford conjectures in characteristic $p$ have been proved by Hrushovski \cite{hrushovsky}, Scanlon \cite{scanlon} and Pink and Roessler \cite{pink}.

In recent work, Serban proved a multiplicative Manin-Mumford-type theorem for multi-variable power series over $p$-adic number fields (Theorem 4.9 in \cite{serban}), recovering an analogue of \cref{thm:ist} with $K$ replaced by a $p$-adic number field and $P$ by a $2$-variable power series convergent on the $2$-dimensional $p$-adic unit disk.
His work does not consider non-Archimedean local fields of positive characteristic and does not provide explicit bounds for the number of torsion points.

We prove a generalization of \cref{cor:complex_case} to Laurent series over non-Archimedean local fields, which are exactly the finite extensions of $\QQ_p$ and the finite extensions of $\FF_p((t))$.
The only condition placed on these Laurent series is that they \emph{converge on the unit sphere}.

\begin{thmx}
		  \label[thm]{thma}
		  Let $K$ be a non-Archimedean local field and let $f\in K[[X,X^{-1}]]$ be convergent on the unit sphere.
		  Assume that the set
		  \begin{align}
					 \label{set}
		  \{\zeta\in\mu_\infty:f(\zeta)\in\mu_\infty\}
		  \end{align}
		  is infinite.
		  If $K$ is of characteristic $0$, then $f(X)=\eta X^m$ for some $\eta\in\mu_\infty$ and $m\in\ZZ$.
		  If $K$ is of positive characteristic $p$, then $f\in\overline\FF_p[X,X^{-1}]$.
\end{thmx}
We call $f$ special if it is of the form specified in the conclusion of \cref{thma}.
A key element of the proof is showing that the equation $f(X)^q-f(X^q)=0$, where $q$ is the residue characteristic of $K$, implies that $f$ is special or has no roots of unity on its graph. 
In characteristic $0$, we need to treat ramified roots of unity separately, while the proof in positive characteristic requires less care.
The treatment of the ramified case builds on work by Schmidt \cite{harry}.

\cref{thm2,thm3} give computable bounds for the size of the set
$$\{\zeta\in\mu_\infty:f(\zeta)\in\mu_\infty\}.$$
Those bounds depend on the numbers of zeroes of absolute value $1$ of $f(X)^q-f(X^q)$ and, in characteristic $0$, also on the ramification degree of $K$ and an analytic computable constant $c_f$ (cf. \cref{lem:bound_p_part}).
The constant $c_f$ is estimated in \cref{c_lemma}
Considering, for example, a non-constant polynomial $P\in K[X]$, the constant $c_P\leq2\deg P$.
\cref{cor:example} provides an effective upper bound on torsion points on the graph of a polynomial $P\in\QQ_p[X]\setminus\ZZ_p[X]$ that only depends on $p$ and the number of roots of $P$ of $p$-adic modulus $1$. 

\section{Preliminaries}
Let $F$ be a complete non-Archimedean field with absolute value $|\cdot|_v$ and let $(a_n)$ be a sequence of elements of $F$.
The Laurent series
\begin{align}
		  \label{eq:f}
f(X)=\sum\limits_{n=-\infty}^{\infty}f_nX^n\in K[[X,X^{-1}]]
\end{align}
converges at a point $z\in F$ if and only if $\lim_{|n|\to\infty}|f_n|_v|z|_v^n=0$.
\begin{defn}[Analytic Laurent series, \cite{cherry}]
		  Let $f$ be as in \cref{eq:f} and $a,b\geq 0$ two real numbers.
		  If $\lim_{|n|\to\infty}|f_n|_vr^n=0$, for all $r\in [a,b]$, then we call $f$ \emph{analytic} on the annulus $A[a,b]$.
		  We denote the set of analytic functions on this annulus by $\mathcal{A}_{[a,b]}$.
		  This set is closed under addition and multiplication.
\end{defn}

For better readability, we will introduce a non-standard definition.
\begin{defn}[Special Laurent series]
		  \label{def:special}
		  Let $K$ be a non-Archimedean field and $f\in \mathcal{A}_K[1,1]$.
		  We call $f$ \emph{special} if either 
		  $K$ is of characteristic $0$ and there are $\eta\in\mu_\infty$ and $m\in\ZZ$ such that
		  $$f(X)=\eta X^m,$$
		  or if $K$ is of characteristic $p>0$ and $f\in\overline{\FF}_p[X,X^{-1}]$.
\end{defn}

Since any non-Archimedean local field is either a finite extension of $\QQ_p$ or of $\FF_p((t))$, these two cases cover all non-Archimedian local fields.
We restate our main result.

%Reset the counter so this theorem is called Theorem A again
\setcounter{thmx}{0}
\begin{thmx}
		  \label[thm]{thm1}
		  Let $K$ be a non-Archimedean local field and $f\in \mathcal{A}_K[1,1]$.
		  Assume that the set
		  \begin{align}
					 \label{eqn:set}
		  \{\xi\in\mu_\infty:f(\xi)\in\mu_\infty\}
		  \end{align}
		  is infinite.
		  Then $f$ is special.
\end{thmx}

This classification cannot be improved. 
Indeed, for any special $f$ and for any $\zeta\in\mu_\infty$, we have $f(\zeta)\in\mu_\infty$.

\subsection{Non-Archimedean Analysis}
\label{rigid_analysis}
This section introduces the required analytic tools and notation.
The exposition follows \cite{cherry}.\\
These results will usually be in the setting of a field $F$ that is complete with respect to a non-Archimedean absolute value $|\cdot|_v$.
\cref{thm:weierstrass_prep} however also requires $F$ to be algebraically closed.
The following is standard in characteristic $0$, but also holds in positive characteristic.
\begin{thm}[Appendix A, Lemma 6 in \cite{bosch}]
		  \label[thm]{thm:alg_closure_completion}
		  Let $F$ be a non-Archimedean local field with absolute value $|\cdot |_v$ and $\Omega$ be an algebraic closure of $F$.
		  Then the completion $\Omega_{|\cdot |_v}$ of $\Omega$ with respect to $|\cdot |_v$ is algebraically closed.
\end{thm}

Thus, every non-Archimedean local field is contained in an algebraically closed field that is complete with respect to a non-Archimedian absolute value extending $|\cdot|_v$.
For a $p$-adic number field, we will fix such a field and denote it by $\CC_p$. 
For a non-Archimedean local field of positive characteristic, we will denote the analogue construction by $\Omega_K$.

We can define a set of non-Archimedean absolute values on the ring of analytic functions on such an annulus. 

\begin{defn}[The absolute values $|\cdot|_r$]
		  \label{def:abs_value_r}
		  Let $0\leq a\leq 1\leq b$ be real numbers and $r\in [a,b]$. 
		  For any Laurent series $f=\sum_n f_n\in \mathcal{A}_{[a,b]}$ let 
		  $$ |f|_r:= \sup\limits_{n\in\ZZ}|f_n|_v r^n.$$
\end{defn}
\begin{defn}[$K(f,r)$ and $k(f,r)$]
		  Let $f,r$ be as in \cref{def:abs_value_r} and additionally $r>0$.
		  We define
		  $$K(f,r) = \sup\{n\in\ZZ||f_n|_vr^n=|f|_r\},$$
		  and similarly
		  $$k(f,r) = \inf\{n\in\ZZ||f_n|_vr^n=|f|_r\}.$$
		  For better readability, we will call any coefficient $f_n$ with $|f_n|_vr^n=|f|_r$ \emph{dominant} on the sphere of radius $r$.
\end{defn}

\begin{lem}[\cite{cherry}]
		  Let $f,r$ be as in \cref{def:abs_value_r} and let $g$ be another Laurent series in $\mathcal{A}_{[a,b]}$.
		  Then $K(fg,r)=K(f,r)+K(g,r)$ and $k(fg,r)=k(f,r)+k(g,r)$.
\end{lem}

$K(f,r)$ and $k(f,r)$ are related with the valuation polygon and Newton polygons of $f$ (cf. \cite{cherry} or \cite{roberts}).
Let $P\in F[X]$ be a polynomial and $r\geq 0$ a real number with $|z|_v=r$ for some $z$.
Then by the ultrametric inequality $P(x)$ cannot have roots on the sphere of radius $r$ unless $K(P,r)>k(P,r)$.
Indeed, the following strengthening holds.

\begin{lem}[2.3.4 in \cite{cherry}]
		  A non-zero polynomial $P$ has $K(P,r)-k(P,r)$ zeroes, counting multiplicity, of absolute value $r$.
\end{lem}

An application of this yields a version of the Weierstrass preparation theorem.

\begin{thm}[Weierstrass Preparation, 2.4.3 in \cite{cherry}]
		  \label[thm]{thm:weierstrass_prep}
		  Let $K$ be algebraically closed and complete with respect to a non-Archimedean absolute value $|\cdot|_v$.
		  Let $0\leq a\leq 1\leq b$ with $r\in[a,b]$ be real numbers and $f\in\mathcal{A}_{[a,b]}$.
		  Let $d:=K(f,r)-k(f,r)$. Then
		  $$ f = Pu$$
		  for a unique pair $(P,u)$ of a polynomial $P\in K[X]$ of degree $d$ with $P(0)=1$, $K(P,r)=d$, $k(P,r)=0$ and a $u\in\mathcal{A}_{[a,b]}$ with $k(u,r)=K(u,r)$.
\end{thm}
\begin{rmk}
		  For our use case, this implies the more algebraic versions of the Weierstrass preparation theorem (e.g. in \cite{lang_algebra}). 
		  Indeed, for all $r\geq 0$ where $K(f,r)\neq k(f,r)$, \cref{thm:weierstrass_prep} gives a polynomial $P_r$.
		  There are only finitely many such $r$ in the unit ball (cf. \cite{cherry}).
		  Then $\prod_{r\in[0,1)}P_r$ is a 'distinguished' polynomial as defined in the algebraic versions of this theorem.
\end{rmk}
As a direct corollary, we can derive the location (the distance from the origin) of zeroes of an analytic function from the absolute value of its coefficients.
\begin{cor}
		  \label[cor]{cor:weierstrass_prep}
		  Let $K,f,r$ be as in \cref{thm:weierstrass_prep}.
		  Then, counting multiplicity, $f$ has exactly $K(f,r)-k(f,r)$ zeroes of absolute value $r$ in $K$.
\end{cor}

\subsection{$p$-adic number fields}
We provide some unsurprising results about algebraic extensions of non-Archimedean local fields for later use.

The ramification degree of an extension stays stable when passing to an unramified extension of $K$.
For an extension $L/K$ denote its \emph{maximal unramified subextension} (cf. II.7.4 in \cite{neukirch}) by $L^{nr(K)}$.
\begin{lem}
\label[lem]{lem:ram_degree_stable}
    Let $K/\Qp$ be a $p$-adic number field and $\eta$ a primitive $m$-th root of unity with $(m,p)=1$. Then
    $$[K(\eta):K(\eta)^{nr(\QQ_p)}]=[K:K^{nr(\QQ_p)}].$$
\end{lem}
\begin{proof}
    Denote by $e,f$ the ramification index and inertia degree of $K/\Qp$.
    Since $\Qp$ is a perfect field, the fundamental equality $[K:\Qp]=ef$ holds (II.6.8 in \cite{neukirch}).
    Denote by $\lambda/\kappa/\FF_p$ the residue fields of $K(\eta)/K/\Qp$.
    The extension $K(\eta)/K$ is unramified (II.7.12 in \cite{neukirch}).
    Hence, 
    \begin{align}
    \label{eqn:eta_tower}
        [K(\eta):\QQ_p]=[K(\eta):K]\cdot[K:\QQ_p]=[\lambda:\kappa]\cdot f\cdot e.
    \end{align}
    The residue fields of $K(\eta)^{nr(\QQ_p)}/K^{nr(\QQ_p)}$ are $\lambda/\kappa$ (II.7.8 in \cite{neukirch}). Therefore
    \begin{align*}
        [K(\eta):\QQ_p]&=[K(\eta):K(\eta)^{nr(\QQ_p)}]\cdot[K(\eta)^{nr(\QQ_p)}:K^{nr(\QQ_p)}]\cdot[K^{nr(\QQ_p)}:\QQ_p]\\
        &=[K(\eta):K(\eta)^{nr(\QQ_p)}]\cdot[\lambda:\kappa]\cdot [\kappa:\FF_p]\\
        &=[K(\eta):K(\eta)^{nr(\QQ_p)}]\cdot[\lambda:\kappa]\cdot f.
    \end{align*}
    Combining this with \cref{eqn:eta_tower} above yields $[K(\eta):K(\eta)^{nr(\QQ_p)}]=e=[K:K^{nr(\QQ_p)}]$.
\end{proof}
\begin{lem}
\label[lem]{lem:galois_isometry}
    Let $L/K$ be a Galois extension of non-Archimedean local fields with absolute values $|\cdot|_L$ extending $|\cdot|_K$ and $\sigma\in\Gal(L/K)$.
    Then $\sigma$ is an isometry.
\end{lem}
\begin{proof}
    The function $|\cdot|_{\sigma}:=|\sigma(\cdot)|_L$ defines an absolute value on $L$ that restricts to $|\cdot|_K$ on $K$. But since the extension of $|\cdot|_K$ to $L$ is unique (II.6.2 in \cite{neukirch}), $|x|_{\sigma}=|x|_L$ for all $x\in L$.
\end{proof}
\begin{lem}
\label[lem]{lem:pth_roots_degree}
    Let $K$ be a $p$-adic number field that is unramified over $\QQ_p$.
    Let $k$ be a positive integer and $\zeta_{p^k}$ be a primitive $p^k$-th root of unity. Then 
    $$[K(\zeta_{p^k}):K]=\varphi(p^k),$$
    where $\varphi$ denotes the Euler totient function.
\end{lem}

\begin{proof}
    Because $K$ is unramified, $p$ stays prime in the ring of integers of $K$.
    One proves the irreducibility of the $p^k$-th cyclotomic polynomial by an Eisenstein argument, e.g. verbatim as in II.7.13 in \cite{neukirch}. 
\end{proof}

\section{Theorem A}
\label{sec:thm1}

The proof of \cref{thm1} proceeds in two parts. 
First, \cref{lem:bound_p_part} bounds the $p$-part of the rank of torsion points on the graph of $f$.
Afterwards, passing to a set of auxiliary Laurent series allows us to restrict any further analysis to unramified roots of unity, those whose order is not divisible by $p$.

\subsection{The ramified part}
\label{sec:ramified}

In positive characteristic the order of a root of unity is never divisible by $p$, so we only need to investigate the case of $p$-adic number fields (finite extensions of $\QQ_p$) here.
We first establish a technical lemma, proving the existence and specifying the form of the constant $c_f$ that appears in \cref{lem:bound_p_part}.

\begin{lem}
		  \label[lem]{c_lemma}

		  Let $K$ be a complete, algebraically closed non-Archimedean valued field and let $f=\sum_jf_jX^j\in K[[X,X^{-1}]]$ be a non-zero Laurent series that is analytic on the unit sphere.
		  Let $k_1,k_2\in\ZZ$, not both equal to zero, and $B\in\ZZ_{>0}$ such that $|k_1|,|k_2|\leq B$ for some $B>0$.
		  Then for any $b\in K$ with $|b|=1$, the Laurent series
		  $$a(X):=f(X^{k_1})-bX^{k_2}$$
		  is either equal to zero or has at most $c_fB$ zeroes of absolute value $1$.
		  The invariant $c_f$ depends only on $f$ and is given as follows.
		  \begin{compactenum}[(i)]
		  \item If $|f|_1>1$, then $c_f = K(f,1)-k(f,1)$.
		  \item If $|f|_1<1$, then $c_f = 0$.
		  \item If $|f|_1=1$ and $K(f,1)\neq k(f,1)$, then $c_f = 2\max(|K(f,1)|,|k(f,1)|)$.
		  \item If $|f|_1=1$ and $K(f,1) = k(f,1)$, then 
					 $$c_f = \begin{cases}
							  2\max(|K(f,1)|,1) &\text{ if $\tilde f=0$ }\\
							  2\max(|K(f,1)|,| K(\tilde f,1)|,| k(\tilde f, 1)|,1) &\text{ otherwise }
					 \end{cases}
					 ,
					 $$
		  \end{compactenum}
		  where $$\tilde f:=\sum\limits_{j\neq k(f,1)}f_jX^j $$
		  is $f$ with its first dominant summand removed. 
\end{lem}
		  Note that $c_f$ does not change when multiplying any of the coefficients of $f$ by any unit of $\OO_K$ because
		  $K(\cdot,1),k(\cdot,1)$ only depend on the absolute values of coefficients.

\begin{proof}
		  By \cref{cor:weierstrass_prep}, this amounts to bounding $K(a,1)-k(a,1)$.

We first treat case (i), where $|f|_1>1$.
			For all dominant coefficients $f_j$ of $f(X^{k_1})$ we have $|f_j|>1$, so in particular $|f_j-b|=|f_j|$.
			Thus, $K(a,1)-k(a,1)=K(f(X^{k_1}),1)-k(f(X^{k_1}),1)=|k_1|(K(f,1)-k(f,1))$.

			We now treat case (ii), where $|f|_1<1$.
		  We have $|a_{k_2}|\geq |f_j-b|=\max (|f_j|,1)=1$ for all $f_j$ and in particular $K(a,1)-k(a,1)=k_2-k_2=0$, so $a$ has no zeros of absolute value $1$.

		  We now treat case (iii), where $|f|_1=1$ and $K(f,1)-k(f,1)>0$.
		  If $k_1\neq0$, there are at least two coefficients of $f(X^{k_1})$ with absolute value $1$ and $bX^{k_2}$ can't cancel both of them. 
		  If $k_1=0$, by assumption $bX^{k_2}$ does not cancel the dominant coefficient either. Thus, $|a|_1=1$ and 
		  \begin{align}
					 \label{c_1}
		  k(a,1)\geq \min(k(f(X^{k_1}),1),k_2). 
		  \end{align}
		  Similarly we obtain
		  \begin{align}
\label{c_2}
		  K(a,1)\leq \max(K(f(X^{k_1}),1),k_2). 
		  \end{align}
		  $k_1$ could have a negative sign and lead to $k_1K(f,1)<k_1k(f,1)$.
		  Thus, we obtain
		  \begin{align}
					 \label{c_3}
		  k(f(X^{k_1})=\min(k_1K(f,1),k_1k(f,1)),
		  \end{align}
		  and similarly
		  \begin{align}
					 \label{c_4}
		  K(f(X^{k_1})=\max(k_1K(f,1),k_1k(f,1)).
		  \end{align}
		  Combining \cref{c_1,c_2,c_3,c_4} leads to
		  \begin{align*}
					 K(a,1)-k(a,1)&\leq\max(K(f(X^{k_1}),1),k_2)-\min(k(f(X^{k_1}),1),k_2)\\
									  &\leq 2 \max(|k_1K(f,1)|,|k_1k(f,1)|,|k_2|).
		  \end{align*}

We finally treat case (iv), where $|f|_1=1$ and $K(f,1)-k(f,1)=0$.
If $k_1k(f,1)\neq k_2$, we can use the same argument (and bound) as in case (iii). 
		  Otherwise, the $k_2$-th coefficient of the Laurent series $a$ is $f_{k(f,1)}-b$.
		  Since by assumption $|f_{k(f,1)}|=|b|$ holds, this coefficient can be arbitrarily small.

		  Define 
			$$\tilde f:=\sum\limits_{j\neq k(f,1)}f_jX^j $$
			If $\tilde f=0$, either $a(X)=0$, in which case $f(X^{k_1})=bX^{k_2}$ or $a(X)=cX^k_2$ and it has no zeroes on $A[1,1]$.

			Otherwise, we distinguish between the following cases.
			\begin{compactenum}[(a)]
			\item If $|f_{k(f,1)}-b| > |\tilde f|_1$, then $f_{k(f,1)}-b$ stays dominant and $K(a,1)=k(a,1)=k_2$.
			\item If $|f_{k(f,1)}-b| = |\tilde f|_1$, then $$K(a,1)=\max(|k_2|,|k_1||K(\tilde f,1)|,|k_1||k(\tilde f)|).$$
			\item If $|f_{k(f,1)}-b| < |\tilde f|_1$, then $$K(a,1)-k(a,1)=|k_1|(K(\tilde f,1)-k(\tilde f,1)).$$
			\end{compactenum}

\end{proof}

\cref{lem:bound_p_part} adapts the proof strategy of Corollary 1 in \cite{harry} to the $p$-adic setting.
We cite an auxiliary result from the same source.
The proof of this Lemma is based on a pigeon-hole principle argument.

\begin{lem}[Lemma 3 with $n=2$ in \cite{harry}]
		  \label[lem]{lem:harry_rou}
		  Let $(\zeta_1,\zeta_2)\in\mu_\infty^2$ be of order $N$.
		  There exists an integer $e\leq N^ \frac{3}{4}$ dividing $N$ and a primitive $N$-th root of unity $\zeta_N$ such that $\zeta_i=\zeta_e^{(i)}\zeta_N^{k_i},i=1,2$, where $|k_i|\leq N^{ \frac{3}{4} }/e$ and $\zeta_e^{(i)}$ is an $e$-th root of unity for $i=1,2$.
\end{lem}

\begin{prop}
		  \label[prop]{lem:bound_p_part}
		  Let $K$ be a $p$-adic number field and $f\in\mathcal{A}_K[1,1]$.
		  Assume that $f$ is not special.
		  If $(\zeta,f(\zeta))$ is a torsion point of $(K^\times)^2$ of order $mp^k$ with $p\nmid m$, then $p^k$ is bounded by a constant only depending on $f$ and the ramification degree of $K$. More precisely,
		  $$p^k\leq \left(2c_f[K:K^{nr(\QQ_p)}]\right)^4$$
		  where $c_f$ is the constant obtained in \cref{c_lemma}.
\end{prop}

		  Following the strategy of \cite{harry}, we construct for such a torsion point an auxiliary Laurent series in $K[[X,X^{-1}]]$ that vanishes at a primitive root of unity $\xi_{p^k}$.
		  That series has at most $c_fp^{k\frac{3}{4}}$ zeroes of modulus $1$.
		  Then, by taking Galois orbits of $\xi_{p^k}$, we bound the number of zeroes from below in a way that grows faster in $p^k$, obtaining a contradiction for large $p^k$.\\ 

\begin{proof}
		  We write $$(\zeta,f(\zeta
          ))=(\zeta_m^{(1)}\zeta_{p^k}^{(1)},\zeta_m^{(2)}\zeta_{p^k}^{(2)})$$ where $\zeta_n^{(k)}$ denotes an $n$-th root of unity.
		  By \cref{lem:harry_rou} there is a $p$-th power $e$, two $e$-th roots of unity $\xi_e^{(1)},\xi_e^{(2)}$ and a primitive $p^k$-th root of unity $\xi_{p^k}$ such that 
		  $$(\zeta_{p^k}^{(1)},\zeta_{p^k}^{(2)})=(\xi_e^{(1)}\xi_{p^k}^{k_1},\xi_e^{(2)}\xi_{p^k}^{k_2})$$
		  with $|k_i|_\infty\leq p^{ \frac{3}{4}k}$ for $i=1,2$ and at least one of the $k_i$ not $0$.

		  Fix a primitive $m$-th root of unity $\eta_m$ and a primitive $e$-th root of unity $\eta_e$.
		  We define the auxiliary power series
		  $$a(X) := \zeta_m^{(2)}\xi_e^{(2)} X^{k_2} - f(\zeta_m^{(1)}\xi_e^{(1)} X^{k_1})\in K(\eta_m,\eta_e)[[X,X^{-1}]].$$
		  Note that $a(\xi_{p^k})=0$ by design.

		  Passing to $\CC_p$, we are in the setting of \cref{c_lemma} and, unless $f$ is special, we can bound the number of zeros of $a$ from above in a way that only depends on $f$ and $p^k$.
		  Write $$K(a,1)-k(a,1)\leq c_f p^{ \frac{3}{4}k }$$
		  with $c_f$ as in \cref{c_lemma}.

		  We will now bound the number of Galois conjucates of $\xi_{p^k}$ from below. Recall that by \cref{lem:ram_degree_stable} 
		  $$[K(\eta_m):K(\eta_m)^{nr(\QQ_p)}]=[K:K^{nr(\QQ_p)}],$$ 
		  and in particular this degree is independent of $m$.
		  Then $$[K(\eta_m,\eta_e):K(\eta_m)^{nr(\QQ_p)}]\leq [K:K^{nr(\QQ_p)}]e.$$
		  On the other hand, by \cref{lem:pth_roots_degree}
		  $$[K(\eta_m)^{nr(\QQ_p)}(\xi_{p^k}):K(\eta_m)^{nr(\QQ_p)}]=(p-1)p^{k-1}.$$
		  In particular, there are 
		  $$[K(\eta_m,\xi_{p^k}):K(\eta_m,\eta_e)]\geq \frac{(p-1)p^{k-1}}{[K:K^{nr}]e} $$
		  Galois conjugates of $\xi_{p^k}$ over $K(\eta_m,\eta_e)$.
		  By \cref{lem:galois_isometry}, $\sigma a(x)=a(\sigma x)$ for all $\sigma\in\Gal(\bar K/K)$, so $a$ has at least that many zeros in $\CC_p$.

		  Let $N$ be the number of zeros of $a$ in $\CC_p$.
		  Then, combining both bounds, we obtain 
		  $$\frac{(p-1)p^{k-1}}{[K:K^{nr}]e}\leq N \leq  \frac{c_f\cdot p^{\frac{3}{4}k}}{e}.$$
		  But for large enough $k$ the left side is larger than the right side.
		  Indeed if there is a torsion point of order divisible by $p^k$ for some $k$, we must have
		  $$(p-1)p^{ \frac{k}{4}-1}\leq c_f [K:K^{nr}].$$
		  Since $p\geq 2$, we have $p-1\geq \frac{p}{2}$ and finally obtain 
		  $$ p^k\leq (2c_f[K:K^{nr}])^4.$$
\end{proof}

Note that the analogous step in the complex case in \cite{harry} is performed simultaneously over all primes. 
However, the available Galois lower bounds in the $p$-adic case are not strong enough for a similar treatment here.
For every finite set $S$ of primes, one can obtain sufficient bounds for $m$ of the form $m=\prod_{p\in S} p^k$, as for example done in section $6$ in \cite{harry_galois}.
But using this would only result in finiteness of torsion points of order $m$ of the above form for each $S$.

\subsection{The unramified part}
\label{sec:unramified}

Knowing that the $p$-part of the order of a torsion point on the graph is bounded, we can now focus on those points of order coprime to $p$ of the auxiliary power series $f(\zeta_{p^k}^{(j)}X)^{p^k}$ (with $\zeta_{p^k}^{(j)}$ ranging over all $p^k$-th roots of unity).
The following lemma gives part of the justification for that approach.
From now on, set $\mu_\infty^{nr}:=\{\zeta\in\mu_\infty:(p,\ord(\zeta))=1\}$.

\begin{lem}
		  \label[lem]{lem:bound_aux}
		  Let $K$ be a non-Archimedean local field of residue characteristic $p$ and $f\in \mathcal{A}_K[1,1]$.
		  Assume that $f$ is not special.
		  Then there is a $k\in\ZZ_{\geq 0}$ such that for any primitive $p^k$-th root of unity $\zeta_{p^k}$ we have
		  $$\#\{\zeta\in\mu_\infty:f(\zeta)\in\mu_\infty\}= \#\{(j,\xi)\in\{1,\dots,p^k\}\times\mu_\infty^{nr}:f(\zeta_{p^k}^j\xi)^{p^k}\in\mu_\infty^{nr}\}.$$
\end{lem}

\begin{proof}
		  In positive characteristic $\mu_\infty^{nr}=\mu_\infty$. 
		  Thus, choosing $k=0$ concludes the proof in this case.
		  Assume now that $K$ is a $p$-adic number field.
		  Suppose that there is a $(\zeta_1,\zeta_2)\in\mu_\infty^2$ such that $f(\zeta_1)=\zeta_2$.
		  By \cref{lem:bound_p_part} there is a $k\in\NN$ such that $v_p(\ord(\zeta_1,\zeta_2))\leq k$.
		  Hence, every such point can be (uniquely) written as $(\zeta_{p^k}^{(1)}\zeta_{m}^{(1)},\zeta_{p^k}^{(2)}\zeta_{m}^{(2)})$ for suitable $\zeta_{p^k}^{(i)}\in\mu_{p^k},\zeta_m^{(i)}\in\mu_\infty^{nr}$.
		  Then $$f\left(\zeta_{p^k}^{(1)}\zeta_m^{(1)}\right)^{p^k}=\left(\zeta_m^{(2)}\right)^{p^k}\in\mu_\infty^{nr}.$$
		  Now fix a primitive $p^k$-th root of unity $\zeta_{p^k}$.
		  Define 
		  $$\tilde f^{(j)}(X):=f(\zeta_{p^k}^jX)^{p^k}.$$
		  Then the map
		  $$\phi:\{\zeta\in\mu_\infty:f(\zeta)\in\mu_\infty\}\to \{(j,\xi)\in\{1,\dots,p^k\}\times\mu_\infty^{nr}:\tilde f^{(j)}(\xi)\in\mu_\infty^{nr}\},$$
		  given for coprime $p,\ord(\zeta_m)$ by
		  $$\zeta_{p^k}^j\zeta_m\mapsto (j,\zeta_m),$$
		  is injective. Since $\tilde f^{(j)}(\xi)\in\mu_\infty^{nr}$ implies that $f(\zeta_{p^k}^j\xi)\in\mu_\infty$ and hence that $(j,\xi)=\phi(\zeta_{p^k}^j\xi)$, the map $\phi$ is also surjective.
\end{proof}

The Galois group of an unramified extension $L/K$ is canonically isomorphic to the Galois group of the residue fields and the lift of the Frobenius is well-behaved on unramified roots of unity.

\begin{prop}[II.7.12 in \cite{neukirch}]
		  \label[prop]{lem:frob_lift}
		  Let $L/K$ be an unramified extension of non-Archimedean local fields with residue fields $\lambda/\kappa$ and assume $L=K(\zeta)$ for a root of unity $\zeta$.
		  Assume that the residue field $\kappa$ is of cardinality $q$.
		  Then $\Gal(L/K)$ contains an element $\sigma$ acting on any root of unity $\zeta$ of order coprime to $q$ by $\sigma(\zeta)=\zeta^q$.
\end{prop}

The proof of \cref{thm1} hinges on identifying those elements $f\in\mathcal{A}_K[1,1]$ that satisfy $f(X)^q=f(X^q)$ for $q$ being the cardinality of the residue field of $K$.
The following two results show that these $f$ are indeed special or, in characteristic 0, otherwise map no root of unity to a root of unity. 
We treat characteristic $0$ and positive characteristic separately.

\begin{lem}
\label[lem]{lem:classification_char0}
Let $K$ be a $p$-adic number field and $f\in \mathcal{A}_K[1,1]$ a Laurent series. Assume that $q$ is a positive power of $p$ and that there are $\zeta,\zeta'\in\mu_\infty$ such that $f(\zeta)=\zeta'$. If
$$f(X)^q=f(X^q),$$
then $f$ is special. 
\end{lem}
\begin{proof}
	We will construct a sequence of torsion points on the graph whose order has an unbounded $p$-part. 
	By \cref{lem:bound_p_part}, $f$ must then be special.\\
By iterating the condition, we get that for every $n\in\NN$ 
$$f(X)^{q^n}=f(X^{q^n}).$$
For every $n$ there is a root of unity $\zeta_n\in\CC_p$ of order divisible by $q^n$ such that $\zeta_n^{q^n}=\zeta$.
Since $$\zeta'=f(\zeta_n^{q^n})=f(\zeta_n)^{q^n},$$
we deduce $f(\zeta_n)\in\mu_\infty$.
By \cref{lem:bound_p_part}, we conclude that $f$ is special.
\end{proof}

In characteristic $p$, the coefficients of $f^q$ take a convenient form.

\begin{lem}
		  \label[lem]{lem:q_power}
			Let $K$ be a non-Archimedean local field of characteristic $p>0$, $q$ be a positive power of $p$. 
			Assume that $f\in \mathcal{A}_K[1,1]$ is a Laurent series of the form
			$$f(X)=\sum\limits_{j\in\ZZ}f_jX^j.$$
			Then, taking the $q$-th power of $f$ yields			
			$$f(X)^q=\sum\limits_{j\in\ZZ}f_j^qX^{qj}.$$
\end{lem}

\begin{proof}
		  For any integer $k$ we can write
$$f(X)=\sum\limits_{j<k} f_{j}X^{j}+f_kX^k+\sum\limits_{j>k} f_{j}X^{j}.$$
Taking $q$-th powers and exploiting that the characteristic of $\mathcal{A}_K[1,1]$ divides $q$, we obtain 
$$f(X)^q =\left(\sum\limits_{j<k} f_{j}X^{j}\right)^q+f_k^qX^{kq}+\left(\sum\limits_{j>k} f_{j}X^{j}\right)^q.$$
Comparing coefficients, the statement follows.
\end{proof}

\begin{lem}
\label[lem]{lem:classification_charp}
Let $K$ be a non-Archimedean local field of characteristic $p>0$, $q$ be a positive power of $p$ and $f\in \mathcal{A}_K[1,1]$ a Laurent series. 
If 
$$f(X)^q=f(X^q),$$
then $f\in\FF_q[X,X^{-1}]$. In other words, $f$ is special.
\end{lem}

\begin{proof}
Write $f(X):=\sum_j f_j X^j$.
We apply \cref{lem:q_power} to obtain
$$\sum\limits_{j\in\ZZ}f_j^qX^{jq}=f(X)^q=f(X^q)=\sum\limits_{j\in\ZZ}f_jX^{jq}$$
and in particular $f_j^q=f_j$ for all $j\in\ZZ$.
Since $f$ converges on the annulus $A_K[1,1]$, for  all but finitely many $j\in\ZZ$ we have $|f_j|<1$.
Thus, for all but finitely many $j$ the equation $f_j^q=f_j$ implies $f_j=0$, so $f\in K[X,X^{-1}]$.

But $f_j^q=f_j$ also implies that all $f_j$ are roots of the polynomial $X^q-X\in\FF_p[X]$, so they are contained in $\FF_q$.
\end{proof}

The converse is also true, so this characterization is exhaustive.
The next lemma proves that $f$ is special if and only if the auxiliary series $\tilde f^{(j)}$ is special.

\begin{lem}
		  \label[lem]{lem:aux_series}
		  Let $K$ be a non-Archimedean local field of residue characteristic $p$ and $f\in \mathcal{A}_K[1,1]$.
		  Let $q$ be a positive power of $p$. 
		  Then $f$ is special if and only if $f^q$ is special.
\end{lem}
\begin{proof}
		  In positive characteristic, this lemma immediately follows from\linebreak 
		  \cref{lem:q_power}.
		  Assume that $K$ is a $p$-adic number field and assume that $f$ is special. Then so is $f^q$. 
		  For the converse, assume that $f(X)^q=\eta X^m$ for $\eta\in\mu_\infty$ and $m\in\ZZ$.
		  Note that $m=K(f^q,1)=k(f^q,1)=qK(f,1)=qk(f,1)$, so in particular, $m=nq$ for some $n\in\ZZ$.
		  Let $\zeta_1,\dots,\zeta_q$ be all roots of $X^q-\eta$ in $\overline K$.
		  Thus, $\zeta_1X^n,\dots,\zeta_qX^n$ are all roots of the polynomial $Y^q-\eta X^m\in \mathcal{A}_K[1,1][Y]$.
		  In particular, $f(X)=\zeta_k X^n$ for some $k\in\{1,\dots,n\}$.
\end{proof}

We are now ready to prove \cref{thm1}.

\begin{proof}[Proof of \cref{thm1}]
		  By \cref{lem:bound_aux} there is a $k\in\ZZ_{\geq 0}$ and a primitive $p^k$-th root of unity $k$ such that we only need to count pairs of unramified roots of unity on the graph of the auxiliary Laurent series
		  $$\tilde f^{(j)}:=f(\zeta_k^jX)^{p^k}$$
		  for $j=1,\dots,p^k-1,p^k$.
		  Let $m$ be any natural number coprime to $p$ and $\zeta_m$ be any primitive $m$-th root of unity.
		  Let $q$ be the cardinality of the residue field of $K(\zeta_{p^k})$. 
		  Recall from \cref{lem:galois_isometry} that every $\sigma\in\Gal(K(\zeta_m\zeta_{p^k})/K(\zeta_{p^k}))$ is an isometry. Hence, for every $z\in K(\zeta_m\zeta_{p^k})$ we have
		  $$\sigma(\tilde f^{(j)}(z)) = \tilde f^{(j)}(\sigma(z)).$$
		  By \cref{lem:frob_lift} there is a $\sigma\in\Gal(K(\zeta_m\zeta_{p^k})/K(\zeta_{p^k}))$ such that for every pair $(\zeta,\tilde f^{(j)}(\zeta))\in\mu_m^2$ we have $\sigma(\zeta)=\zeta^q$ and $\sigma(\tilde f^{(j)}(\zeta))=\tilde f^{(j)}(\zeta)^q$. 
		  Therefore
		  $$\tilde f^{(j)}(\zeta)^q = \tilde f^{(j)}(\zeta^q)$$
		  and thus $\zeta$ is a zero of absolute value $1$ of the Laurent series
		  $$g^{(j)}(X):=\tilde f^{(j)}(X)^q-\tilde f^{(j)}(X^q).$$
		  If $g^{(j)}$ equals zero for some $j$, then,  in characteristic $p>0$, by \cref{lem:classification_charp},  $\tilde f^{(j)}$ is special.
          In characteristic $0$, by \cref{lem:classification_char0}, $\tilde f^{(j)}$ is special or $\tilde f^{(j)}(\zeta)\not\in \mu_\infty$ for all $\zeta\in\mu_\infty$, in which case we are done.
          Assume now, in any characteristic, that $\tilde f^{(j)}$ is special.
		  Then so is $f^{p^k}$ and by \cref{lem:aux_series} so is $f$.

		  If, on the other hand, all $g^{(j)}$ are non-zero, these Laurent series each have only finitely many zeroes on the unit sphere and hence the number of pairs of roots of unity on the graph of each $\tilde f^{(j)}$ is finite.
		  This implies that the number of roots of unity on the graph of $f$ is finite.
\end{proof}

\section{Effective bounds}

\cref{thm2,thm3} give bounds on the cardinality of the set $\{\zeta\in\mu_\infty:f(\zeta)\in\mu_\infty\}$. 
The bounds are computable if one has a good grasp on (the absolute values of) the coefficients of $f$ and $f^q$.

In characteristic $p$, the statement is simpler than in characteristic $0$, since there are no ramified roots of unity.

\begin{thm}
		  \label[thm]{thm2}
		  Let $K$ be a non-Archimedean local field of characteristic $p>0$ and with residue field of cardinality $q$.
		  Assume that $f\in \mathcal{A}_K[1,1]$ is not special.
		  Denote by $M$ the number of zeroes of absolute value $1$ in $\Omega_K$ of the Laurent series $f(X)^q-f(X^q)$.
		  Then 
		  $$\#\{\zeta\in\mu_\infty:f(\zeta)\in\mu_\infty\}\leq M.$$
\end{thm}
Note that by \cref{lem:classification_charp}, this $M$ is finite.
\begin{proof}
		  Because there are no $p$-th roots of unity in non-Archimedean local fields of characteristic $p$, the $k$ in $p^k$ in the proof of \cref{thm1} can be chosen to be $k=0$.
		  Thus, bounding the numbers of pairs of roots of unity on the graph of $f$ can be achieved by bounding the number of zeroes of absolute value $1$ in $\Omega_K$ of $f(X)^q-f(X^q)$.
\end{proof}

For $p$-adic number fields, we have to take ramification into account.

\begin{thm}
		  \label[thm]{thm3}
		  Let $K$ be a $p$-adic number field with residue field of cardinality $q$.
		  Assume that $f\in \mathcal{A}_K$ is not special.
		  Let $k\in\ZZ_{\geq 0}$ be such that $p^{k+1}>(2c_f[K:K^{nr}])^4$ where $c_f$ is as in \cref{c_lemma}.
		  Let $M$ be the number of zeroes of absolute value $1$ in $\CC_p$ of the Laurent series
		  $$f(X):=f(X)^{qp^k}-f(X^q)^{p^k}.$$
		  Then 
		  $$\#\{\zeta\in\mu_\infty:f(\zeta)\in\mu_\infty\}\leq p^k M.$$
\end{thm}
Note that by \cref{lem:classification_char0}, this $M$ is finite or the set is empty.

\begin{proof}
		  Combining \cref{lem:bound_p_part,lem:bound_aux}, we now only need to count pairs of unramified roots of unity on the graph of each $\tilde f^{(j)}(x):=f(\zeta_{p^k}^jx)$ for $j=1,\dots,p^k$.
		  Following the proof of \cref{thm1}, each pair of roots of unity on the graph corresponds to a zero of one of the  
		  $$g^{(j)} := \tilde f^{(j)}(X^q)-f^{(j)}(X)^q.$$
		  The auxiliary function $g^{(j)}$ has exactly $M$ zeroes of absolute value $1$ in $\CC_p$, since translation by a unit induces a bijection on $\CC_p$.
		  Thus, there are at most $p^k\cdot M$ pairs of roots of unity on the graph of $f$.
\end{proof}
For some classes of functions, this yields bounds that are independent of the degree.
The following result generalizes the observation that for a polynomial $P$ with exactly one largest $p$-adic coefficient of modulus $m>1$, the equality $|P(x)|_p=m$ holds for all $x\in\CC_p$ of absolute value $1$. In particular, $P(x)$ is never a root of unity and $P$ has no roots of modulus $1$.

\begin{cor}
\label{cor:example}
	Let $K$ be a $p$-adic number field with absolute value $|\cdot|$ and ring of integers $\mathcal{O}_K$ and let $P\in K[X]\setminus\mathcal{O}_K[X]$ be a polynomial with $n\geq 1$ roots of $p$-adic absolute value $1$ (counting multiplicity) in $\CC_p$.
	Then
		  $$\#\{\xi\in\mu_\infty:P(\xi)\in\mu_\infty\}\leq 2^8[K:K^{nr(\QQ_p)}]^8pn^9.$$
\end{cor}
\begin{proof}
	This is an application of \cref{thm3}.
	Note that $n=K(F,1)-k(F,1)$. 
	Furthermore $|F|_1>1$ because $F\not\in\mathcal{O}_K[X]$.
	\cref{lem:bound_aux} gives $c_F = n$ and one may choose the $k$ in \cref{thm3} to be 
	$$k:=\lfloor4\log_p(2n[K:K^{nr(\QQ_p)}])\rfloor.$$
	Next, \cref{thm3} requires bounding the number of roots of absolute value $1$ of the polynomial $A(X):=F(X)^{p^{k+1}}-F(X^q)^{p^k}$.
	Since 
	$$\left|F(X)^{p^{k+1}}\right|_1=|F|_1^{p^{k+1}}>|F|_1^{p^{k}}=\left|F(X^q)^{p^k}\right|_1,$$
	we obtain 
	$$K(A,1)-k(A,1)=K(F(X)^{p^{k+1}},1)-k(F(X)^{p^{k+1}},1)=p^{k+1}n.$$
	We conclude by multiplying that expression with $p^k$. Indeed
	$$p^{2k+1}n\leq pn\cdot p^{8\log_p(2n[K:K^{nr(\QQ_p)}])}=2^8[K:K^{nr(\QQ_p)}]^8pn^9.$$
\end{proof}

\section{Acknowledgments}
This paper originated as a Master's Thesis at the University of Basel.
The author thanks Harry Schmidt for his excellent supervision.
Furthermore, his encouragement and continued support, even long after leaving Basel, were essential in making this research public.
\printbibliography
\end{document}